\documentclass[11pt]{amsart}

\pdfoutput=1
\usepackage{amsmath,amsthm,amssymb,amsfonts, graphicx, color, verbatim, tikz}
\usepackage[nobysame]{amsrefs}

\newtheorem{thm}{Theorem}[section]

\newtheorem{prop}[thm]{Proposition}
\newtheorem{lem}[thm]{Lemma}

\theoremstyle{definition}

\newtheorem{ex}[thm]{Example}

\theoremstyle{remark}
\newtheorem{rmk}[thm]{Remark}
\newtheorem{fact}[thm]{Fact}

\tikzset{node distance=2cm, auto}

\newcommand{\mfm}{\mathfrak{m}}
\newcommand{\mfa}{\mathfrak{a}}
\newcommand{\mfb}{\mathfrak{b}}

\newcommand{\bn}{\mathbb{N}}

\newcommand{\Hom}{\operatorname{Hom}}
\newcommand{\Ext}{\operatorname{Ext}}
\newcommand{\ass}{\operatorname{Ass}}

\newcommand{\depth}{\operatorname{depth}}

\newcommand{\ann}{\operatorname{ann}}

\newcommand{\bit}{\begin{itemize}}
\newcommand{\eit}{\end{itemize}}
\newcommand{\ben}{\begin{enumerate}}
\newcommand{\een}{\end{enumerate}}
\newcommand{\bpf}{\begin{proof}}
\newcommand{\epf}{\end{proof}}

\begin{document}
\title{Systems of parameters and the Cohen-Macaulay property}

\author{Katharine Shultis}
\address{Department of Mathematics, University of Nebraska, Lincoln, NE 68588-0130, USA}
\email{s-kshulti1@math.unl.edu}

\date{\today}


\thanks{This research was partially supported by NSF grant DMS-1201889 and U.S. Department of Education grant P00A120068 (GAANN)}

\begin{abstract}
Let $R$ be a commutative, Noetherian, local ring and $M$ an $R$-module. Consider the module of homomorphisms $\Hom_R(R/\mfa,M/\mfb M)$ where $\mfb\subseteq\mfa$ are parameter ideals of $M$. When $M=R$ and $R$ is Cohen-Macaulay, Rees showed that this module of homomorphisms is always isomorphic to $R/\mfa$, and in particular, a free module over $R/\mfa$ of rank one. In this work, we study the structure of such modules of homomorphisms for general $M$.
\end{abstract}

\maketitle

\section{Introduction}

Let $R$ be a commutative, Noetherian, local ring. This work concerns the module of homomorphisms $\Hom_R(R/\mfa,R/\mfb)$ where $\mfa$ and $\mfb$ are parameter ideals of $R$ with $\mfb\subseteq\mfa$. 

An immediate consequence of a result of Rees \cite{Rees} is that when $R$ is Cohen-Macaulay, this module of homomorphisms is isomorphic to $R/\mfa$. In particular, as an $R/\mfa$-module, it is free of rank one. The focus of this work is to study the structure of this module of homomorphisms when $R$ is not Cohen-Macaulay. Our main results identify circumstances under which it is decomposable and not free.

When $R$ has dimension one and depth zero and $\mfa$ and $\mfb$ are in sufficiently high powers of the maximal ideal, we prove that $\Hom_R(R/\mfa,R/\mfb)$ is neither indecomposable nor free as an $R/\mfa$-module.

We can extend the result about decomposability both to modules and to higher dimensions. In particular, for $M$ a nonzero, finitely generated $R$-module we consider the module $\Hom_R(R/\mfa,M/\mfb M)$ where $\mfb\subseteq\mfa$ are parameter ideals of $M$. When $M$ is not Cohen-Macaulay, we can show that the module of homomorphisms decomposes for any parameter ideal $\mfa$ and for $\mfb$ chosen to be generated by suitable powers of any system of parameters generating $\mfa$. This result generalizes recent work of K. Bahmanpour and R. Naghipour \cite{BandN} in the case where $M=R$. Specifically, when $R$ is not Cohen-Macaulay, they showed there exist some parameter ideals $\mfb\subseteq\mfa$ of $R$ for which $\Hom_R(R/\mfa,R/\mfb)$ is not cyclic.

\section{Preliminaries}

Throughout, $R$ will be a commutative, Noetherian, local ring with maximal ideal $\mfm$, and $M$ will be a finitely generated $R$-module of dimension $d$. A {\it system of parameters} of $M$ is a set of $d$ elements generating an ideal $\mfa$ such that $M/\mfa M$ has finite length. An ideal $\mfa$ generated by a system of parameters is called a {\it parameter ideal.} We begin by reviewing the consequence of Rees' result in the Cohen-Macaulay case.

\begin{rmk}\label{applyRees}
If $M$ is a Cohen-Macaulay $R$-module of dimension $d$, and $\mfb\subseteq\mfa$ are parameter ideals of $M$, then $$\Hom_R(R/\mfa,M/\mfb M)\cong M/\mfa M.$$ In particular, when $M=R$, this is a free $R/\mfa$-module of rank one and hence indecomposable. 

Indeed, the elements of a system of parameters of $M$ form an $M$-regular sequence. The isomorphism above is deduced from Rees' Theorem \cite[Lemma 1.2.4]{BandH}: 
\begin{align*}
\Hom_R(R/\mfa,M/\mfb M) &\cong \Ext_R^d(R/\mfa,M)\\
&\cong \Hom_R(R/\mfa,M/\mfa M)\\
&\cong M/\mfa M.
\end{align*}
\end{rmk}

We now recall some well-known results.

\begin{fact}\label{sopofMorR}
Let ${\bf a}=a_1,\ldots,a_d\in\mfm$ and ${\bf\overline{a}}=\overline{a_1},\ldots,\overline{a_d}$ be the images of the $a_i$ in $R/\ann(M)$. Then ${\bf a}$ is a system of parameters of $M$ if and only if ${\bf\overline{a}}$ is a system of parameters of the ring $S:=R/\ann(M)$.
\end{fact}

The integer $n$ appearing in the next result plays a key role in the main results. We include a proof for completeness. 

\begin{lem}\label{B&NLemmaExtn}
Let $(R,\mfm)$ be a local Noetherian ring and $M$ a finitely generated $R$-module. 
There exists an integer $n$ such that $\mfm^nM\cap \Gamma_\mfm(M)=(0)$.
\end{lem}

\begin{proof}
Since $\Gamma_\mfm(M)$ is Artinian, the descending chain of submodules $$(\mfm M\cap \Gamma_\mfm(M))\supseteq(\mfm^2M\cap \Gamma_\mfm(M))\supseteq\cdots$$ must stabilize. That is, there is some $n\in\bn$ such that $$\mfm^{n+i}M\cap \Gamma_\mfm(M)=\mfm^nM\cap \Gamma_\mfm(M)$$ for all integers $i\geq 0$. Thus
\begin{align*}
\mfm^nM\cap \Gamma_\mfm(M) &=\bigcap_{i\geq n}(\mfm^iM\cap \Gamma_\mfm(M))\\
&\subseteq\bigcap_{i\geq n}\mfm^iM\\
&= (0)\end{align*} by the Krull Intersection Theorem. 
\end{proof}

\begin{rmk}
In fact, given any finite-length submodule $L\subseteq M$, we have $\mfm^nM\cap L=(0)$ where $n$ is the integer of Lemma \ref{B&NLemmaExtn}.
\end{rmk}

The next result is in the spirit of \cite[Prop 4.7.13]{Northcott}. We include a proof in order to obtain specific bounds on the powers of $a$ in this special case.

\begin{prop}\label{arlessmoduleprop}
Let $A$ be any commutative ring, $L$ an $A$-module, and $a,b\in A$. Then, for arbitrary positive integers $p\leq q\leq r$ we have the equality $$(ba^rL:a^p)=a^{r-q}(ba^qL:a^p)+(0:_La^p).$$
\end{prop}

\begin{proof}
First let $x\in(ba^rL:a^p)$. Then $a^px=ba^ry$ for some $y$ in $L$. Now $a^p(x-ba^{r-p}y)=0$ so that $x-ba^{r-p}y\in(0:_La^p)$. Moreover $$ba^{r-p}y=a^{r-q}\cdot ba^{q-p}y\in a^{r-q}(ba^qL:a^p).$$ We now have $$x=ba^{r-p}y+(x-ba^{r-p}y)\in a^{r-q}(ba^qL:a^p)+(0:_La^p).$$ For the other inclusion, let $x\in a^{r-q}(ba^qL:a^p)+(0:_La^p)$ and write $x=a^{r-q}y+z$ with $y$ in $(ba^qL:a^p)$ and $z$ in $(0:_La^p)$. We can write $a^py=ba^qw$ for some $w$ in $L$. Thus we may rewrite $a^px$ as 
\begin{align*}a^px &= a^p(a^{r-q}y+z)\\
&=a^{r-q}ba^qw+a^pz\\
&=ba^rw,
\end{align*}
which is in $ba^rL$. Thus $x\in(ba^rL:a^p)$ as desired.
\end{proof}

We will use the next result in Sections 3 and 4.

\begin{lem}\label{mylem}
Let $R$ be a local Noetherian ring, $J\subseteq I$ ideals of $R$ with $\sqrt I=\sqrt J$, and $N$ a nonzero finitely generated $R$-module. If $\Hom_R(R/J,N)$ is decomposable, then so is $\Hom_R(R/I,N)$.
\end{lem}

\begin{proof}
Suppose that $\Hom_R(R/J,N)= X\oplus Y$ where $X$ and $Y$ are nonzero $R$-modules. There are isomorphisms
\begin{align*}\Hom_R(R/I,N)&\cong\Hom_R((R/I)\otimes_R(R/J),N)\\
&\cong\Hom_R(R/I,\Hom_R(R/J,N))\\
&\cong\Hom_R(R/I,X)\oplus\Hom_R(R/I, Y).
\end{align*}
By symmetry, it suffices to show that $\Hom_R(R/I,X)\neq 0$. It is clear that $JX=(0)$ since $X\subseteq\Hom_R(R/J,N)$. 
Choose $P\in\ass_RX$ such that $J\subseteq P$. Since $\sqrt{J}=\sqrt{I}$, one has $I\subseteq P$, so there are maps $$R/I\twoheadrightarrow R/P\hookrightarrow X.$$ The composition of these maps is nonzero, and so $\Hom_R(R/I,X)\neq (0)$ as desired.
\end{proof}

This final result will be used in Section 4.

\begin{lem}\label{nonCMstability}
Let $R$ be a local Noetherian ring and $M$ a $R$-module of dimension $d\geq 2$. If $M$ is not Cohen-Macaulay, then for any system of parameters $a_1,\ldots,a_d$ of $M$, there exist positive integers $i$ and $s$ such that $M/a_i^sM$ is not Cohen-Macaulay.
\end{lem}

\begin{proof}
If some $a_i$ is $M$-regular, then $M/a_iM$ is not Cohen-Macaulay, so we may assume that each $a_i$ is a zero-divisor on $M$. Suppose, by way of contradiction that $M/a_1^sM$ is Cohen-Macaulay for each $s\geq 1$. Then $a_2,\ldots,a_d$ is a regular sequence on $M/a_1^sM$ for all integers $s\geq 1$. In particular $a_2$ is $M/a_1^sM$-regular for all integers $s\geq1$. We claim this implies $a_2$ is $M$-regular, which is a contradiction. Indeed, suppose $a_2m=0$ for some $m\in M$. Then $a_2\overline{m}=0$ in $M/a_1^sM$ for all integers $s\geq 1$ so that $m\in a_1^sM$ for all integers $s\geq 1$. By the Krull Intersection Theorem, we have $m=0$ implying $a_2$ is $M$-regular, a contradiction.
\end{proof}

\section{Dimension One}

We start with results on modules of dimension one and depth zero since we are able to obtain stronger results in this case. We show $\Hom_R(R/\mfa,M/\mfb M)$ will decompose if the parameter ideal $\mfb$ is chosen to be in a sufficiently high power of the maximal ideal.

\begin{thm}\label{dim1stability}
Let $(R,\mfm)$ be a local Noetherian ring, $M$ a nonzero finitely generated $R$-module of dimension one and depth zero, and $n$ an integer such that $\mfm^nM\cap\Gamma_\mfm(M)=(0)$. For any parameter $a$ of $M$, and any parameter $b$ of $M$ with $b\in(a^{n+1})$, the following $R$-module is decomposable: $$\Hom_R(R/(a),M/bM).$$
\end{thm}

\begin{rmk}The integer $n$ in the statement exists by Lemma \ref{B&NLemmaExtn}. Note that $n\geq 1$ because $\Gamma_\mfm(M)\neq (0)$.
\end{rmk}

\begin{proof}[Proof of Theorem \ref{dim1stability}]
Set $S:=R/\ann(M)$ and let $(\,\bar{\hspace{0.05in}}\,)$ denote the image in $S$. In light of Fact \ref{sopofMorR}, $\overline{a},\overline{b}$ are parameters of $S$. Moreover there is an $R$-module isomorphism
$$\Hom_S(S/(\overline{a}),M/\overline{b}M)\cong\Hom_R(R/(a),M/bM).$$ Thus, by replacing $R$ with $S$, we may assume that $M$ is faithful as an $R$-module.

Write $b=ca^{n+1}$. Since $M$ is faithful, we have $\sqrt{(a)}=\mfm$ and so $$(0:_Ma)\subseteq\Gamma_{(a)}(M)=\Gamma_\mfm(M).$$ Thus we know \begin{equation}\label{intiszero}
(0:_Ma)\cap ca^nM\subseteq \Gamma_\mfm(M)\cap\mfm^nM=(0).
\end{equation}

By Proposition \ref{arlessmoduleprop} we have
\begin{equation}\label{messysum}(ca^{n+1}M:a)=a^{n}(caM:a)+(0:_Ma).\end{equation} We now claim that $$a^{n}(caM:a)=ca^{n}M.$$ 
Indeed, it is clear that elements of $ca^nM$ are also elements of $a^n(caM:a)$. For the reverse inclusion, let $x\in a^{n}(caM:a)$ and write $x=a^nm$ for some $m\in(caM:a)$. So we have $am=cam'$ for some $m'\in M$. Then
$$x=a^nm=a^{n-1}\cdot am=ca^nm'\in ca^nM.$$ Equation (\ref{messysum}) now becomes \begin{equation}\label{sum}(ca^{n+1}M:a)=c a^{n}M+(0:_Ma).\end{equation}

Next we want to show \begin{equation}\label{int}ca^{n+1}M=ca^{n}M\cap\left[(0:_Ma)+ca^{n+1}M\right].\end{equation} 
Elements in $ca^{n+1}M$ also live in both $ca^nM$ and $(0:_Ma)+ca^{n+1}M$. For the other inclusion, let $x\in ca^{n}M\cap[(0:_Ma)+ca^{n+1}M]$ and write $$x=ca^{n}m=\eta+ca^{n+1}m'$$ for some $m,m'\in M$, and $\eta\in(0:_Ma)$. Then \begin{align*}\eta&=ca^{n}m-ca^{n+1}m'\\
&\in(0:_Ma)\cap ca^{n}M = (0) & \text{ by }(\ref{intiszero}).
\end{align*} Equation (\ref{int}) follows. Now there are isomorphisms
\begin{align*}
\Hom_R(R/(a),M/bM) &\cong \frac{(bM:a)}{bM}\\
&\cong \frac{ca^{n}M+(0:_Ma)}{ca^{n+1}M} & \text{ by }(\ref{sum})\\
&\cong \frac{ca^{n}M}{ca^{n+1}M}\oplus\frac{(0:_Ma)+ca^{n+1}M}{ca^{n+1}M} & \text{ by }(\ref{int}).
\end{align*}
All that remains to prove is that both summands are nonzero. 

If the summand on the left were zero, then $ca^{n}M=(0)$ by Nakayama's Lemma, a contradiction as $ca^{n+1}=b$ is a parameter of $M$. 

If the summand on the right were zero, then $$(0:_Ma)\subseteq ca^{n+1}M.$$ By Equation (\ref{intiszero}) we have $$(0:_Ma)=(0:_Ma)\cap ca^{n+1}M=(0).$$ This is also a contradiction as $\depth_RM=0$. Thus $\Hom_R(R/(a),M/bM)$ is decomposable, as desired.
\end{proof}

When $R$ is a Cohen-Macaulay ring, we know from Remark \ref{applyRees} that the  $R/\mfa$-module $\Hom_R(R/\mfa,R/\mfb)\cong R/\mfa$ is not only indecomposable, but also free. When $R$ is one dimensional and not Cohen-Macaulay, we can prove that in addition to decomposing, this module will be non-free if the parameters are chosen to be in sufficiently high powers of the maximal ideal.

\begin{thm}\label{nonfreestabthm}
Let $(R,\mfm)$ be a Noetherian local ring of dimension one and depth zero, and $n$ an integer such that $\mfm^n\cap\Gamma_\mfm(R)=(0)$. For any parameter $a$ in $\mfm^n$ and any parameter $b$ in $(a^2)$, the $R/(a)$-module $$\Hom_R(R/(a),R/(b))$$ is decomposable and has a non-free summand.
\end{thm}

\begin{rmk}Again, the integer $n$ in the statement exists by Lemma \ref{B&NLemmaExtn} and must be positive since $\Gamma_\mfm(R)\neq (0)$.
\end{rmk}

\begin{proof}[Proof of Theorem \ref{nonfreestabthm}]
We will first prove that the module decomposes. Both the proof of this fact and the decomposition obtained are similar to those found in the proof of Theorem \ref{dim1stability}. Write $I=\Gamma_\mfm(R)$. For any $x\in\mfm^n$, we know $(x)\cap I=(0)$ and hence $xI=(0)$. If $x\in\mfm^n$ is also a parameter, then we know $\Gamma_{(x)}(R)=I$, and $(0:x)=I$ as well. Indeed, $\sqrt{(x)}=\mfm$ and since $xI=0$ we have $$I\subseteq (0:x)\subseteq\Gamma_{(x)}(R)=I.$$

Let $a\in\mfm^n$ and $b\in(a^2)$ be parameters and write $b=ca^2$. Applying Proposition \ref{arlessmoduleprop} with $p=q=1$ and $r=2$, we obtain the equality \begin{equation}\label{freestabsumugly}((ca^2):a)=a((ca):a)+(0:a).\end{equation} We now note that $a((ca):a)=(ca).$ We may thus rewrite Equation (\ref{freestabsumugly}) as 
\begin{equation}\label{freestabsum}
((b):a)=(ca)+I.
\end{equation} 
Next we wish to show that
\begin{equation}\label{freestabint}
(ca^2)=(ca)\cap\left[I+(ca^2)\right].
\end{equation}
The inclusion $\subseteq$ is clear. For the other inclusion, let $x\in (ca)\cap[I+(ca^2)]$ and write $x=rca=\eta+r'ca^2$ for some $r,r'\in R$ and $\eta\in I$. Then
\begin{align*}
\eta&=rca-r'ca^2\\
&\in (a)\cap I\\
&\subseteq \mfm^n\cap I=(0).
\end{align*}
 Hence $\eta=0$ and $x=r'ca^2\in (a^2)$. Hence there are isomorphisms of $R/(a)$-modules
\begin{align*}
\Hom_R(R/(a),R/(b)) &\cong \frac{((b):a)}{(b)}\\
&\cong \frac{(ca)+I}{(ca^2)} & \text{ by }(\ref{freestabsum})\\
&\cong \frac{(ca)}{(ca^2)}\oplus\frac{I+(ca^2)}{(ca^2)} & \text{ by }(\ref{freestabint}).
\end{align*}
Next we show that both summands are nonzero. 

If the summand on the left were zero, then Nakayama's Lemma implies $ca=0$, a contradiction as $ca^2=b$ is a parameter and hence nonzero. 

If the summand on the right were zero, then $I\subseteq (ca^2)$ so that $$I=I\cap(ca^2)\subseteq I\cap (a)=(0),$$ a contradiction as the depth of $R$ is zero. 

We now show that the summand on the left, that is $(ca)/(ca^2)$, is not a free $R/(a)$-module. To that end, recall that $I\cap (a)=(0)$, but $I\neq (0)$ so we can choose an element $y\in I\setminus(a)$. Since $aI=(0)$, $\overline{y}$ is a nonzero element of $R/(a)$ that annihilates $(ca)/(ca^2)$, and hence $(ca)/(ca^2)$ cannot be free as an $R/(a)$-module.
\end{proof}

\section{Higher Dimensions}

In higher dimensions, we can also prove a decomposition theorem. However, Example \ref{cantextendex} shows that Theorem \ref{dim1stability} is not strong enough to use the induction technique in Theorem \ref{higherdimdec} to prove there is an integer $N$ such that $\Hom_R(R/\mfa,R/(a_1^{n_1},\ldots,a_d^{n_d}))$ decomposes for all $n_i\geq N$. K. Bahmanpour and R. Naghipour \cite{BandN} prove that when $R$ is not Cohen-Macaulay $\Hom_R(R/\mfa,R/\mfb)$ is not cyclic for some parameter ideals $\mfa$ and $\mfb$ with $\mfb\subseteq\mfa$.

\begin{thm}\label{higherdimdec}
Let $R$ be a local Noetherian ring and $M$ a finitely generated $R$-module of dimension $d$. If $M$ is not Cohen-Macaulay, then, for any system of parameters ${\bf a}=a_1,\ldots,a_d$ of $M$, there exist integers $n_1,\ldots,n_d\in\bn$ such that the following $R$-module is decomposable: $$\Hom_R(R/({\bf a}),M/(a_1^{n_1},\ldots,a_d^{n_d})M).$$
\end{thm}

\begin{proof}
As in the proof of Theorem \ref{dim1stability}, we may reduce to the case that $M$ is a faithful module. We proceed by induction on $d$, the case $d=1$ being covered by Theorem \ref{dim1stability}.

Assume, now, that $d\geq 2$. By Lemma \ref{nonCMstability}, we can find some $i\leq d$ and a positive integer $n_i$ such that $M/a_i^{n_i}M$ is not Cohen-Macaulay. We may harmlessly assume $i=1$. Set $$\overline{R}:=R/(a_1^{n_1})\qquad\text{and}\qquad\overline{\mfa}:=(\overline{a_2},\ldots,\overline{a_d}).$$ Then $\overline{\mfa}$ is a parameter ideal of $\overline{M}$. Since $\overline{M}$ has dimension $d-1$, by induction there are natural numbers $n_2,\ldots, n_d$ such that the $R$-module
 $$U:=\Hom_{\overline{R}}(\overline{R}/\overline{\mfa},\overline{M}/(\overline{a_2}^{n_2},\ldots,\overline{a_d}^{n_d})\overline{M})$$ decomposes. Since there is an isomorphism $$U\cong\Hom_R(R/(a_1^{n_1},a_2,\ldots,a_d),M/(a_1,a_2^{n_2},\ldots,a_d^{n_d})M),$$ then applying Lemma \ref{mylem}, gives the desired decomposition.
\end{proof}

The result below is a version of Theorem \ref{nonfreestabthm} for rings of arbitrary dimension.

\begin{thm}\label{higherdimnonfree}
Let $R$ be a local Noetherian ring of dimension $d$. If $R$ is not Cohen-Macaulay, then for any system of parameters $a_1,\ldots,a_d$ of $R$, there exist integers $n_1,\ldots,n_d,N_1,\ldots,N_d$ with $n_i\leq N_i$ for $i=1,\ldots,d$ such that the $R/(a_1^{n_1},\ldots,a_d^{n_d})$-module $$\Hom_R(R/(a_1^{n_1},\ldots,a_d^{n_d}),R/(a_1^{N_1},\ldots,a_d^{N_d}))$$ is decomposable and not free.
\end{thm}

\begin{proof}
We proceed by induction on $d$. If $d=1$, then choosing $n_1$ to be the $n$ from Theorem \ref{nonfreestabthm} and $N_1=2n_1$ works.

Now suppose that $d\geq 2$. By Lemma \ref{nonCMstability}, we can find integers $i$ and $n_i$ such that $R/(a_i^{n_i})$ is not Cohen-Macaulay. We may harmlessly assume $i=1$. Set $S:=R/(a_1^{n_1})$ let $(\,\bar{\hspace{0.05in}}\,)$ denote the image in $S$. Then $\overline{a_2},\ldots,\overline{a_d}$ is a system of parameters of $S$ and, by induction, there exist integers $n_2,\ldots,n_d,N_2,\ldots,N_d$ such that the $S/(\overline{a_2}^{n_2},\ldots,\overline{a_d}^{n_d})$-module $$U:=\Hom_S(S/(\overline{a_2}^{n_2},\ldots,\overline{a_d}^{n_d}),S/(\overline{a_2}^{N_2},\ldots,\overline{a_d}^{N_d}))$$ decomposes and has a non-free summand. Note that $$S/(\overline{a_2}^{n_2},\ldots,\overline{a_d}^{n_d})\cong R/({a_1}^{n_1},\ldots,{a_d}^{n_d}).$$ Setting $N_1=n_1$ we then have $$U\cong\Hom_R(R/(a_1^{n_1},\ldots,a_d^{n_d}),R/(a_1^{N_1},\ldots,a_d^{N_d}))$$ and this gives the desired decomposition and non-free summand.
\end{proof}

\section{Examples}

In this section, we focus on examples. In particular, we investigate the structure of the $R/\mfa$-module $\Hom_R(R/\mfa,R/\mfb)$ for concrete examples of $R$, $\mfa$, and $\mfb$. 

Let us take $M=R$ in Theorem \ref{higherdimdec}. If we take $n_i=1$ for each $i$ then \begin{equation}\label{niisone}\Hom_R(R/(a_1,\ldots,a_d),R/(a_1^{n_1},\ldots,a_d^{n_d})\cong R/(a_1,\ldots,a_d)\end{equation} is a free $R/(a_1,\ldots,a_d)$-module of rank one. Our first example shows that Equation \ref{niisone} sometimes holds even when $R$ is not Cohen-Macaulay and at least one of the $n_i$'s is greater than one.

\begin{ex}
Consider the parameter $y$ of $R=k[[x,y]]/(x^2,xy^2)$. Then we have $$\Hom_R(R/(y),R/(y^2))\cong\frac{(y^2):_Ry}{(y^2)}=\frac{(y)}{(y^2)}\cong R/(y).$$
\end{ex}

The next example shows that the module $\Hom_R(R/\mfa,R/\mfb)$ can be neither cyclic nor decomposable and also that the bound in Theorem \ref{dim1stability} is not always optimal.

\begin{ex}
Consider the parameter $y^2$ of $R=k[[x,y]]/(x^2,xy^m)$. Then $$U_t:=\Hom_{R}(R/(y^2),R/(y^{2t}))$$ is $$\begin{cases}
\text{cyclic}, & \text{if } t<\frac{m+1}{2},\\
\text{indecomposable, but not cyclic}, & \text{if } t=\frac{m+1}{2}, \text{ and}\\
\text{decomposable}, & \text{if } t>\frac{m+1}{2}.
\end{cases}$$
However, Theorem \ref{dim1stability} only predicts that $U_t$ decomposes for $t>m+1$ since $\mfm^n\cap\Gamma_\mfm(R)\neq0$ for $n<m$.
\end{ex}

The next example shows that even when all of the parameters are zero divisors, $M$ may have positive depth, and $M/aM$ may be Cohen-Macaulay.

\begin{ex}
Consider the ring $R=k[[x,y,z]]/(x^2,xyz)$ of dimension two and depth one along with the system of parameters $y,z$. Both $y$ and $z$ are zero-divisors in $R$ and both $R/(y)$ and $R/(z)$ are Cohen-Macaulay rings of dimension one.
\end{ex}

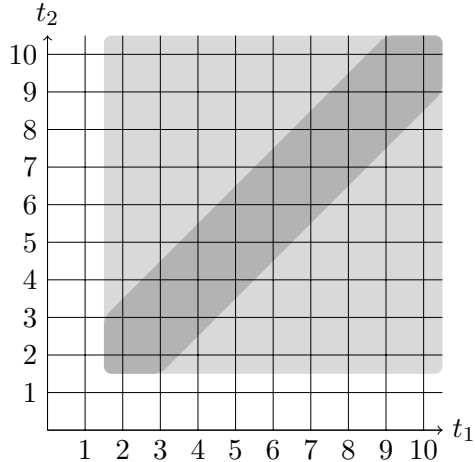
\begin{figure}[htbp]
\label{2Dpicture}
\caption{In this figure, a lattice point $(t_1,t_2)$ corresponds to the module $\Hom_R(R/(y,z),R/(y^{t_1},z^{t_2}))$ from Example \ref{cantextendex}. The modules corresponding to lattice points in the light grey regions are known to decompose due to Theorem \ref{dim1stability}. The modules corresponding to lattice points in the middle dark grey region are known to decompose by direct computation. The modules corresponding to lattice points where $t_1=1$ or $t_2=1$ are indecomposable since $R/(y)$ and $R/(z)$ are both Cohen-Macaulay rings.}
\begin{tikzpicture}[scale=0.5]
    \fill[rounded corners=1mm, gray!30] (1.5,1.5) -- (10.5,1.5) -- (10.5,10.5) -- (1.5,10.5) -- cycle;
    \fill[rounded corners=1mm, gray!60] (1.5,1.5) -- (1.5,3) -- (9,10.5) -- (10.5,10.5) -- (10.5,9) -- (3,1.5) -- cycle;
    \draw[very thin,color=black] (0,0) grid (10.5,10.5);
    \draw[->] (0,0) -- (10.5,0) node[right] {$t_1$};
    \draw[->] (0,0) -- (0,10.5) node[above] {$t_2$};
    \draw[] (1,0) node[below] {$1$};
    \draw[] (2,0) node[below] {$2$};
    \draw[] (3,0) node[below] {$3$};
    \draw[] (4,0) node[below] {$4$};
    \draw[] (5,0) node[below] {$5$};
    \draw[] (6,0) node[below] {$6$};
    \draw[] (7,0) node[below] {$7$};
    \draw[] (8,0) node[below] {$8$};
    \draw[] (9,0) node[below] {$9$};
    \draw[] (10,0) node[below] {$10$};
    \draw[] (0,1) node[left] {$1$};
    \draw[] (0,2) node[left] {$2$};
    \draw[] (0,3) node[left] {$3$};
    \draw[] (0,4) node[left] {$4$};
    \draw[] (0,5) node[left] {$5$};
    \draw[] (0,6) node[left] {$6$};
    \draw[] (0,7) node[left] {$7$};
    \draw[] (0,8) node[left] {$8$};
    \draw[] (0,9) node[left] {$9$};
    \draw[] (0,10) node[left] {$10$};
\end{tikzpicture}
\end{figure}

Theorems \ref{dim1stability} and \ref{nonfreestabthm}, which give bounds on the powers needed to make $\Hom_R(R/\mfa,M/\mfb M)$ decompose and be non-free, apply only in dimension one. However, examples seem to indicate that the $R/\mfa$-module $$\Hom_R(R/\mfa,R/(a_1^{n_1},\ldots,a_d^{n_d}))$$ is neither free nor indecomposable if the $n_i$ are large enough. One such example is explained below.

\begin{ex}\label{cantextendex}
Again, consider the ring $R=k[[x,y,z]]/(x^2,xyz)$ of dimension two and depth one along with the system of parameters $y,z$. If $n_1\geq 2$, then $S_{n_1}:=R/(y^{n_1})$ is not Cohen-Macaulay (since the non-zero element $xy^{n_1-1}$ is in the socle). Letting $\mfm$ be the maximal ideal of $S_{n_1}$ we have $\mfm^n\cap\Gamma_\mfm(S_{n_1})=0$ if and only if $n\geq n_1+2$. By symmetry, the same holds for the ring $T_{n_2}:=R/(z^{n_2})$. Thus Theorem \ref{dim1stability} gives that $U_{n_1,n_2}:=\Hom_R(R/(y,z),R/(y^{n_1},z^{n_2})$ decomposes for all $n_1,n_2\geq 2$ with $|n_1-n_2|>2$. However, direct computation shows that $U_{n_1,n_2}$ actually decomposes for all $n_1,n_2\geq 2$. See Figure 1 for a visual representation of this.
\end{ex}

\section{Acknowledgements}
This work was done as part of my Ph.D. thesis work at the University of Nebraska - Lincoln, under the guidance of my advisors, Srikanth Iyengar and Roger Wiegand. I would like to thank them for their assistance in the development of this work and their suggestions in the writing of this paper.

\end{document}